\documentclass[a4paper]{amsart}

\usepackage{amsmath}
\usepackage{amsthm}
\usepackage{amssymb}
\usepackage{mathrsfs}
\usepackage{scalerel}
\usepackage{mathrsfs}
\usepackage[mathscr]{euscript}
\usepackage{hyperref}

\newtheorem{theorem}{Theorem}[section]
\newtheorem*{theorem*}{Theorem}

\theoremstyle{definition}
\newtheorem{definition}[theorem]{Definition}
\newtheorem{proposition}[theorem]{Proposition}

\newtheorem{observation}[theorem]{Observation}
\newtheorem*{observation*}{Observation}
\newtheorem{fact}[theorem]{Fact}
\newtheorem*{fact*}{Fact}

\newtheorem{corollary}[theorem]{Corollary}
\newtheorem{remark}[theorem]{Remark}
\newtheorem*{remark*}{Remark}

\newtheorem{question}[theorem]{Question}

\theoremstyle{plain}
\newtheorem{lemma}[theorem]{Lemma}

\renewcommand{\tt}{^{(t)}}
\newcommand{\bra}[1]{ \left( #1 \right) }

\newcommand{\abs}[1]{\left|#1\right|}

\newcommand{\norm}[1]{\left\lVert #1 \right\rVert}

\newcommand{\cN}{\mathcal{N}}

\newcommand{\fpa}[1]{\left\lVert #1 \right\rVert_{\mathbb{R}/\mathbb{Z}}}
\newcommand{\e}{\varepsilon}
\renewcommand{\a}{\alpha}

\renewcommand{\b}{\beta}
\newcommand{\de}{\delta}

\newcommand{\Borel}{\mathscr{B}}

\DeclareMathOperator*{\EE}{\mathbb{E}}

\newcommand{\EEl}[2]{\EE_{#1 \leq #2}}

\newcommand{\NN}{\mathbb{N}}
\newcommand{\QQ}{\mathbb{Q}}
\newcommand{\PP}{\mathbb{P}}
\newcommand{\ZZ}{\mathbb{Z}}
\newcommand{\RR}{\mathbb{R}}
\newcommand{\TT}{\mathbb{T}}
\newcommand{\cA}{\mathscr{A}}
\newcommand{\cB}{\mathscr{B}}

\newcommand{\fX}{\mathscr{X}}
\newcommand{\fZ}{\mathscr{Z}}

\newcommand{\fY}{\mathscr{Y}}

\newcommand{\densNat}{\mathrm{d}}

\newcommand{\Mod}[1]{\bmod{#1}}

\newcommand{\set}[2]{\left\{ #1 \ \middle| \ #2 \right\} }

\newcommand{\charN}{\chi}
\newcommand{\charX}[1]{\chi(#1)}
\newcommand{\lc}{\operatorname{lc}}
\newcommand{\spr}[2]{\left< #1,\ #2\right>}

\newcommand{\normHK}[1]{{\left\vert\kern-0.25ex\left\vert\kern-0.25ex\left\vert #1 
    \right\vert\kern-0.25ex\right\vert\kern-0.25ex\right\vert}}

\numberwithin{equation}{section}

\begin{document}

\begin{abstract}

We investigate polynomial patterns which can be guaranteed to appear in \emph{weakly mixing} sets introduced by Furstenberg and studied by Fish. In particular, we prove that if $\cA$ is a weakly mixing set and $p(x) \in \ZZ[x]$ a polynomial of odd degree with positive leading coefficient, then all sufficiently large integers $N$ can be represented as $N = n_1 + n_2$, where $p(n_1) + m,\ p(n_2) + m \in \cA$ for some $m \in \cA$.

\end{abstract}

\baselineskip=17pt

\title[Weakly mixing sets and polynomial equations]{Weakly mixing sets of integers \\ and polynomial equations}

\author[J. Konieczny]{Jakub Konieczny}
\address{Mathematical Institute \\ 
University of Oxford\\
Andrew Wiles Building \\
Radcliffe Observatory Quarter\\
Woodstock Road\\
Oxford\\
OX2 6GG}
\email{jakub.konieczny@gmail.com}

\date{}

\maketitle

\setcounter{section}{0}

\section{Introduction}
It is a fundamental question in additive combinatorics to determine which types of structure are guaranteed to appear in a given set of the integers. We begin with citing the celebrated theorem of Szemer{\'e}di \cite{Szemeredi1969}, whose ergodic theory proof by Furstenberg \cite{Furstenberg1977} paved the way to applications of ergodic theory in combinatorial number theory.
\begin{theorem}[Szemer{\'e}di]\label{thm:Szemeredi}
	Let $\cA \subset \NN$ be a set with positive upper Banach density. 
	Then, for any $k \in \NN$, there exist $n,m \in \NN$ such that $n,\ n+m,\ n + 2m, \dots n+km \in \cA$. 
\end{theorem}

Phrased differently, the theorem asserts that any set of positive density contains arithmetic progressions of arbitrary length. Many generalisations of this theorem exist. A theorem of S\'{a}rk\"{o}zy \cite{Sarkozy1978c} (see also \cite{Furstenberg1977}, \cite{Furstenberg1981}) asserts that in sets of positive upper Banach density one can find patterns such as $n,n+m^2$. In approximately the same time, but different direction, a result of Furstenberg and Katznelson \cite{FurstenbergKatznelson1978} pertains to configurations in higher dimensions, showing that a set $\cA \subset \NN^r$ of positive upper Banach density contains the configuration $n + m F$, where $F \subset \ZZ^r$ is any finite set. 

Returning to the polynomial in a single dimension, Bergelson and Leibman \cite{BergelsonLeibman1996} were able to improve S\'{a}rk\"{o}zy's theorem to several polynomials vanishing at $0$. This result was ultimately strengthened by these authors and Lesigne \cite{BergelsonLeibmanLesigne2007} to deal with {intersective} families of polynomials. A sequence $p_i(x) \in \ZZ[x]$, $i \in [r]$ is \emph{intersective} if for any integer $k$ there exists $n_k \in \NN$ such that $k \mid p_i(n_k)$ for all $i \in [r]$. 

\begin{theorem}[Bergelson, Leibman, Lesigne]\label{thm:Belle}
	Let $\cA \subset \NN$ be a set with positive upper Banach density, and let $p_i(x) \in \ZZ[x]$ for $i \in [r]$ be an intersective family of polynomials with $p_i(n) \to \infty$ as $n \to \infty$. Then, there exists $n,m \in \NN$ such that $m, m+p_1(n), \dots, m+p_{r}(n) \in \cA$.
\end{theorem}

Note that the conclusion of the above theorem fails if $p_i$ are not intersective. Moreover, the offending set $\cA$ can be very structured: indeed, an (infinite) arithmetic progression will do. 

On the other hand, one expects that more can be proved if $\cA$ is forced to be \emph{unstructured}. In the extreme case, when $\cA$ is a random set, constructed by declaring $n \in \cA$ with a certain probability $\rho > 0$, independently for all $n$, then with probability $1$, $\cA$ will contain many occurrences of the pattern, say, $m, m+p_1(n), m+p_2(n), \dots, m + p_r(n)$ for any polynomials (or, for that matter, any functions $p_1,p_2,\dots,p_r$). Thus, it is of interest to see which notions of pseudo-randomness guarantee existence of various patterns.

The class of \emph{weakly-mixing} sets was proposed by Furstenberg and investigated by Fish \cite{Fish2007a}, \cite{Fish2007b}. Roughly speaking, a weakly mixing set is a set of return times of a generic point to a neighbourhood of its origin in a weakly-mixing measure preserving system $\fX_\cA$. While the precise definitions will be given in due course, at this point we remark that weakly mixing sets include \emph{normal} sets, i.e.\ those sets for which any pattern of $0$'s and $1$'s appears in the characteristic sequence of the set with the same frequency as for a genuinely random set.

In \cite{Fish2007a}, Fish characterised all linear patterns which are guaranteed to appear in a weakly mixing set. We give a special (yet representative) case of this characterisation.
\begin{theorem}[Fish \cite{Fish2007a}]\label{thm:Fish-1}
	Let $\cA \subset \NN$ be a weakly mixing set. Suppose that $a_i, b_i \in \NN$ and $c_i \in \ZZ$ for $i \in [r]$ are such that for all $i \neq j$ we have $\det \begin{bmatrix} a_i & a_j \\ b_i & b_j \end{bmatrix} \neq 0$. Then, there exist $n,m \in \NN$ such that $a_i n + b_i m + c_i \in \cA$ for all $i \in [r]$.
\end{theorem}  

For example, a set $\cA$ will contain the pattern $n, m, n+m, n+2m, 2n+m$, which is not guaranteed to appear on the grounds of density alone. However, unlike in the case of a normal set, a weakly-mixing set is not guaranteed to contain two consecutive elements $n,n+1$. When it comes to polynomial patterns, one has a result of a somewhat different flavour, which bears resemblance to Theorem \ref{thm:Belle}.

\begin{theorem}[Fish \cite{Fish2007b}]\label{thm:Fish-2}
	Let $\cA \subset \NN$ be a weakly-mixing set, and let $\cB \subset \NN$ be a set of positive density. Let $p_i(x) \in \ZZ[x]$ for $i \in [r]$ be polynomials of equal degree, such that for all $i \neq j$ we have $\deg(p_i - p_j) > 0$, and $p_i(n) \to \infty$ as $n \to \infty$. Then, for all $n \in \NN$ except\footnote{Here and elsewhere, when a statement is said to hold for ``all $N$ except for a set of $0$ density'', we mean that there exists a set $S \subset \NN$ with density $1$ such that the statement holds for all $N \in S$. The meaning of the phrase ``all but finitely many'' is analogous.}
 for a set of density $0$,  
	there exist $m \in \cB$ such that $p_1(n) - m, p_2(n) - m, \dots , p_r(n) - m \in \cA$.
\end{theorem}

In a previous paper \cite{Konieczny2016a}, the author investigated the question of whether certain sets of polynomial recurrence are bases of positive integers. A representative instance is the following question.

\begin{question}\label{quest:1}
	Fix $\a \in \RR \setminus \QQ$ and a polynomial $p(x) \in \ZZ[x]$ with $p(n) \to \infty$ as $n \to \infty$. For $\e > 0$, let $\cA$ be the Bohr set $\set{ n \in \NN }{ n \a \bmod {1} \in (- \e, \e) }$. Is it the case that for all $\e > 0$, the set $\set{n \in \NN}{p(n) \in \cA}$ is a basis of order $2$ for the positive integers? 
\end{question}

Here, a set $\cB \subset \NN$ is a basis of order $2$ if there exists $N_0 = N_0(\cB)$ such that for $N \geq N_0$, there are $n_1, n_2 \in \cB$ with $N = n_1 + n_2$. Hence, we are asking if, for sufficiently large $N$, there is a solution to $N = n_1 + n_2$ with $p(n_1) \in \cA, \ p(n_2) \in \cA$.

The answer to Question \ref{quest:1} is (trivially) negative in the case when $\deg p = 1$. Somewhat surprisingly, the answer remains negative when $\deg p = 2$ for generic choice of $\a$. Finally, when $\deg p \geq 3$, the answer is positive, again for a generic choice of $\a$. For exact statements, we refer to \cite{Konieczny2016a}. 

This paper arose from an attempt to see what happens at the other extreme, where instead of being structured, the set $\cA$ is weakly-mixing.

When $\deg p = 1$, then it is not a significant loss of generality to assume that $p(x) = x$. The question then becomes: Is any weakly-mixing set $\cA$ a basis of order $2$? 
The answer to this is negative, but $\cA$ is \emph{almost} a basis of order $2$, in the sense  $\cA+\cA$ has density $1$ (see \cite{Fish2007a}). In the case when $\deg p \geq 2$, one cannot expect to guarantee that a weakly-mixing set $\cA$ contains \emph{any} elements from $\set{p(n)}{n \in \NN}$. Indeed, if $\cA$ is weakly-mixing, then so is $\cA \setminus \mathcal{Z}$ for any $\mathcal{Z}$ of density $0$, and thus in particular $\cA \setminus \set{p(n)}{n \in \NN}$ is weakly mixing. However, we are able to prove the following.

\begin{theorem}\label{thm:main}
	Let $\cA \subset \NN$ be a weakly-mixing set, and let $p(x)\in \ZZ[x]$ be a non-constant polynomial. Then, all $N \in \NN$ except for a set of $0$ density
	can be represented as $N = n_1 + n_2$, where  $n_1,n_2 \in \NN$ are such that $p(n_1) + m \in \cA$ and $p(n_2) + m \in \cA$ for some $m \in \cA$. Moreover, if $\deg p$ is odd then the same conclusion holds for all but finitely many $N$.
\end{theorem}

The above theorem is a direct consequence of two more technical results, which may be of independent interest. To formulate them, we need to introduce some terminology. For a polynomial $p$, we denote by $\deg p$ and $\lc p$ the degree and leading coefficient of $p$, respectively, so that the leading term of $p$ is $x^{\deg p} \lc p$. We shall say that a family of polynomials $p_i^{(N)}(x) \in \ZZ[x]$ for $i \in [r],\ N \in \NN$, is \emph{uniformly admissible} if the following conditions hold:
\begin{enumerate}
	\item For each $i$ and $N$, $\deg p_i^{(N)} > 0$, and $\lc p_i^{(N)},\ \deg p_i^{(N)}$ do not depend on $N$. 
	\item For each $i \neq j$, $\deg (p_i^{(N)} - p_j^{(N)}) > 0$, and $\lc (p_i^{(N)} - p_j^{(N)}),\ \deg (p_i^{(N)} - p_j^{(N)})$ do not depend on $N$.
\end{enumerate}

For instance, the pair $p_1^{(N)}(n) = n^3,\ p_2^{(N)}(n) = (N-n)^3$ is uniformly admissible, but the pair  $p_1^{(N)}(n) = n^2,\ p_2^{(N)}(n) = (N-n)^2$ is not. 

\begin{theorem}\label{thm:main-uniform}
	Let $\cA \subset \NN$ be a weakly-mixing set, and let $\cB \subset \NN$ be a set of positive density. Let $p_i^{(N)}(x) \in \ZZ[x]$ for $i \in [r],\ N \in \NN$, be a family of polynomials which is \emph{uniformly admissible}. Then, there exists $N_0$ such that for any $N > N_0$, there are $n \in [N]$ and $m \in \cB$ such that $p_1^{(N)}(n) + m, \dots, p_r^{(N)}(n) + m \in \cA$. 
\end{theorem}

Rather than fully general families of polynomials, we are interested specifically in those which are themselves given by polynomial formulas. In other words, we will consider a sequence $p_i^{(N)}(x) \in \ZZ[x,N]$, $i \in [r]$. We will say that such sequence is \emph{admissible} if the following holds:
\begin{enumerate}
	\item For each $i$, $\deg_x p_i^{(N)}(x) > 0$ (as polynomial in two variables).
	\item For each $i \neq j$, $\deg_x (p_i^{(N)}(x) - p_j^{(N)}(x)) > 0$.
\end{enumerate}
 
For instance, the pair $p_1^{(N)}(n) = n^k,\ p_2^{(N)}(n) = (N-n)^k$ is admissible for each $k \in \NN$.
 
\begin{theorem}\label{thm:main-polynomial}
	Let $\cA \subset \NN$ be a weakly-mixing set, and let $\cB \subset \NN$ be a set of positive density. Let $p_i^{(N)} (x)\in \ZZ[x,N]$ for $i \in [r]$ be a polynomial family of polynomials, which is \emph{admissible}. Then, there exists a set $S \subset \NN$ with density $\densNat(S) = 1$, such that for all $N \in S$, there are $n \in [N]$, $m \in \cB$ such that $p_1^{(N)}(n) + m, \dots, p_r^{(N)}(n) + m \in \cA$. 
\end{theorem}

\begin{proof}[Proof of Theorem \ref{thm:main} assuming \ref{thm:main-uniform} and \ref{thm:main-polynomial}]

	In the case when $\deg p$ is odd, the pair $p_1^{(N)}(x) = p(x),\ p_2^{(N)}(x) = p(N-x)$ is uniformly admissible, hence by Theorem \ref{thm:main-uniform} applied with $\cB = \cA$, for all sufficiently large $N$ there exist $n,m \in \cA$ such that $p(n)+m \in \cA,\ p(N-n)+m \in \cA$. It remains to put $n_1 = n,\ n_2 = N-n$.
	
	In the case when $\deg p$ is even, we apply Theorem \ref{thm:main-polynomial}, and use the fact that the pair $p_1^{(N)}(x), p_2^{(N)}(x)$ just defined (but now viewed as  an element of $\ZZ[x,N]$) is admissible. The remainder of the argument is fully analogous.
\end{proof}

\begin{remark}
	Our proof of Theorem  \ref{thm:main} depends on the parity of $\deg p$ in a crucial way. However, it is not a priori clear that the conclusion of this theorem should depend on $\deg p$. In fact, the author believes that the stronger conclusion (the set of exceptional $N$ being finite) holds also when $\deg p$ is even, but it does not appear to be possible to obtain this result with our methods.
\end{remark}

This paper is organised as follows. In Section \ref{sec:Definitions} we give basic definitions, specifically we define the weakly-mixing sets. In Section \ref{sec:Uniform} we reduce Theorem \ref{thm:main-uniform} to a uniform convergence statement in ergodic theory, and prove a special case of it. In Section \ref{sec:PET} we introduce the PET induction and finish the proof of Theorem \ref{thm:main-uniform}. In Section \ref{sec:Polynomial} we again reduce Theorem \ref{thm:main-polynomial} to a statement about certain ergodic averages, and then prove this statement. Finally, in Section \ref{sec:End} we prove a stronger version of some technical results from \ref{sec:Polynomial}.

This paper draws heavily on the work of Bergelson \cite{Bergelson1987} and Fish \cite{Fish2007a}, \cite{Fish2007b}. Many of the ideas we use can be traced back to their, or earlier, work.

\subsection*{Acknowledgements}

The author thanks Vitaly Bergelson for useful comments on Theorem \ref{ergo::thm:PET-uniform-1}, and Ben Green for advice and support during the work on this project. The author is also grateful to the anonymous referee for noticing an error in an earlier version of this paper. 
Finally, thanks go to Sean Eberhard, Freddie Manners, Rudi Mrazovi\'{c}, Przemek Mazur and Aled Walker for many informal discussions. 
\section{Definitions, convenitions and basics}\label{sec:Definitions}

	Throughout the paper, we denote the characteristic function of a set $X$ by $1_X$. We use the convention $\NN = \{1,2,3, \dots\}$, $\NN_0 = \NN \cup \{0\}$, and $[n] = \{1,2,\dots,n\}$. To simplify notation, we use the symbol $\EE$ borrowed from probability to denote averages: $\EE_{x \in X} f(x) = \frac{1}{\abs{X}} \sum_{x \in X} f(x)$. 

For a set of integers $\cA \subset \NN$ we define its density as $\densNat(\cA) = \lim_{N \to \infty} 
\EE_{n \in [N]} 1_{\cA}(n)$, provided that the limit exists, which will usually be the case in this paper. Upper and lower densities are defined accordingly.

	We shall use standard asymptotic notation. We write $X = O(Y)$ or $X \ll Y$ if $X \leq C Y$ for an absolute constant $C > 0$. If $C$ is allowed to depend on a parameter $M$, we write $X = O_M(Y)$. In presence of a variable $n$, we write $X = o_{n \to \infty} (Y)$ (or simply $X = o(Y)$ if no confusion is possible) if $X/Y \to 0$ as $n \to \infty$. If the rate of convergence is allowed to depend on $M$, we write $X = o_{M;n \to \infty}(Y)$.

\newcommand{\sC}{\mathscr{C}}
A \emph{measure preserving system} $\fX = (X,T,\Borel, \mu)$ consists of a compact metrizable space $X$, together with a probability measure $\mu$ on a Borel $\sigma$-algebra $\cB$, and a $\Borel$-measurable transformation $T \colon X \to X$, such that $\mu(T^{-1}E) = \mu(E)$ for all $E \in \Borel$. The transformation $T$ acts on function on $X$ by composition: $(Tf)(x) = f(Tx)$.

Recall that a m.p.s.\ $\fX$ is \emph{ergodic} if for any $A,B \in \cB$ we have that $\EE_{n \leq N} \mu(A \cap T^{-n}B) \xrightarrow[N \to \infty]{} \mu(A) \mu( B)$, and similarly $\fX$ is 
 \emph{weakly mixing} if we have the stronger condition
$\EE_{n \leq N} \abs{ \mu(A \cap T^{-n}B) - \mu(A)\mu(B) } \xrightarrow[N \to \infty]{}~0.$ A point $x \in X$ is \emph{generic} if for any $f \in C(X)$ one has $\EE_{n \leq N} T^n f(x) \to \int f d\mu$. It is a consequence of the ergodic theorem that $\mu$-almost all points are generic.

 A morphism between m.p.s.'s $\fX = (X,\Borel, \mu, T)$ and $\fY = (Y,\sC, \nu, S)$ is a $(\Borel, \sC)$-measurable map $\pi \colon X \to Y$ such that $\pi \circ T = S \circ \pi$ and $\pi_* \mu = \nu$. In this context, $\fY$ is a \emph{factor} and $\fX$ is an \emph{extension}. Any factor $\fY$ of $\fX$ is uniquely determined, up to isomorphism, by the $\sigma$-algebra $\Borel' \subset \Borel$ generated by the sets $\pi^{-1}(F)$, $F \in \sC$. In particular, we have the conditional expectation operation $\EE(\cdot|\fY)$, which we can view (with obvious identifications) as mapping $L^p(\mu)$ to a subspace of $L^p(\mu)$, $p \in [1,\infty]$.

It will be convenient to view a set $\cA \subset \NN_0$ of positive upper density as arising from dynamics. Let $\Omega := \{0,1\}^{\NN_0}$ denote the \emph{shift space}, taken with the natural product topology and the Borel $\sigma$-algebra. On $\Omega$, we may define the \emph{shift map} given by $(Sx)(i) := x(i+1)$. To $\cA \subset \NN_0$ we can always associate its characteristic function $1_{\cA} \in \Omega$, which gives rise the \emph{subshift} $X_\cA := \mathrm{cl} \set{S^n 1_{\cA}}{ n \in \NN_0}$, which is evidently a closed and $S$-invariant subspace of $\Omega$. 

\begin{definition}
	A set $\cA \subset \NN_0$ with positive upper density is \emph{weakly-mixing} if and only if the point $1_{\cA} \in X_{\cA}$ is generic for some ergodic $S$-invariant probability measure $\mu_{\cA}$ (which is necessarily unique), such that the resulting measure preserving system denoted $\fX_{\cA} = (X_\cA, S, \Borel(X_\cA), \mu_{\cA})$ is weakly mixing.
\end{definition}

 We stress that a weakly mixing set $\cA$ is in particular required to have positive upper density, and it has a density since $\EE_{n\leq N} 1_\cA(n) \xrightarrow[N\to\infty]{} \mu(\set{x \in \Omega}{x(0)=1)}$. A seemingly more general definition of weakly mixing systems is possible.

\begin{observation}\label{obs:def-of-WM-set-general}
	A set $\cA \subset \NN_0$ is weakly mixing if and only if it takes the form $\cA = \set{n \in \NN_0}{ f(T^n x_0) = 1}$, where $\fX = (X,T,\Borel,\mu)$ is a weakly mixing system, $f \in L^\infty(\mu)$ takes values $0$ and $1$, $\int f d\mu > 0$, and $x_0$ is $f$-\emph{generic}. Here, a point $x_0$ is $f$-generic if for any $g$ in the algebra generated by $f,\ Tf,\ T^2f, \dots$, we have the convergence of the averages:
$$
	\EE_{n \leq N} g( T^{n} x_0) \xrightarrow[N \to \infty]{} \int_{X} g d\mu. 
$$ 
\end{observation}
\begin{proof} 
Clearly, any weakly mixing system is of the aforementioned form, with $\fX = \fX_\cA$, so only one implication needs to be proved. Suppose that a set $\cA = \set{n \in \NN}{ f(T^n x_0) = 1}$ is as in the latter definition. Define the measurable map $F\colon X \to \Omega$  given by $x \mapsto ( f(T^n x) )_{n \in \NN_0}$. It is clear that $F \circ T = S \circ F$, and hence the pushforward $\nu := F_* \mu$ is a $S$-invariant measure on $\Omega$. Since $\fY = (\Omega, S, \Borel(\Omega), \nu)$ is a factor of $X$, it is easy to check that $\fY$ is weakly mixing. Note also that $1_{\cA} = F(x_0)$.

It remains to check that $1_{\cA}$ is generic for thus defined $\nu$. It will suffice to verify that for any cylinder $U = \set{x \in \Omega}{ x(0) = \epsilon_0, \dots, x(r) = \epsilon_r }$ it is the case that $\EE_{n \leq N} 1_U( S^n 1_{\cA} ) \to \nu(U)$. But this is an easy consequence of $f$-genericity of $x_0$. Indeed, let $f_i(x) = f(x)$ if $\epsilon_i = 1$ and $f_i(x) = 1-f(x)$ if $\epsilon_i = 0$; then
\begin{align*}
	\EE_{n \leq N} 1_U (S^n 1_{\cA}) = \EE_{n \leq N} \prod_{i \leq r} f_i( T^{n+i} x_0) \xrightarrow[N \to \infty]{} \int_{X} \prod_{i \leq r} T^i f_i d\mu = \nu(U). \qquad  \qedhere
\end{align*}
\end{proof}
We close with a remark on invertible extensions. A m.p.s.'s $\fX = (X,T,\Borel,\mu)$ is \emph{invertible} if $T$ is invertible. Any m.p.s.\ $\fX$ has a canonical invertible extension $\tilde \fX$, and the invertible extension of a weakly mixing system is again weakly mixing. Provided that $T$ is continuous and surjective (as in the case for $\fX_\cA$ mentioned above), we may ensure that if $x \in \fX$ is generic, and $\tilde x$ is a lift of $x$, then $\tilde x$ is generic as well. Hence, for any weakly mixing set $\cA$, we may assume that it originates from an invertible weakly mixing system via the construction in Observation \ref{obs:def-of-WM-set-general} (in simpler terms, we may relate $\cA$ to a two-sided shift rather than the one-sided one used in Definition \ref{obs:def-of-WM-set-general}). Most of the time, we assume invertibitily for the sake of convenience, but our results on ergodic averages, such as Theorem \ref{ergo::thm:PET-uniform-1} and \ref{ergo::thm:PET-uniform-2} remain true for non-invertible systems with minor modifications.  
\renewcommand{\tt}{^{(t)}}
\newcommand{\nn}{^{(N)}}
\renewcommand{\d}{\ \mathrm{d}}

\section{Uniform ergodic theorem}\label{sec:Uniform}

We will now explain how Theorem \ref{thm:main-uniform} can be derived from a result in ergodic theory, concerning convergence of certain averages. Because the set $\cA$ is already related to a m.p.s.\ $\fX_{\cA} = (X_\cA, S, \mu_{\cA})$, it comes as no surprise that we will be interested in averages of functions for this system. 

Fix a family of polynomials $p_{i}\nn$, $i \in [r]$ as in Theorem \ref{thm:main-uniform} or \ref{thm:main-polynomial}, and assume for simplicity that $\cB = \cA$. For large integers $M,N$, let $\cN(N,M)$ denote the number of solutions to  
\begin{equation}
	p_1\nn(n) + m \in \cA, \quad \dots, \quad p_r\nn(n) + m \in \cA, \quad m \in \cA 
	\label{eq:001}
\end{equation}
with $n \in [N],\ m \in [M]$. Let also $f_0 \in C(X_{\cA})$ be the function given by $f_0(x) = x(0)$. We may then approximate, at least heuristically:
\begin{align*}
	\frac{1}{NM}\cN(N,M) &= \EE_{m \leq M} \EE_{n \leq N} 1_{\cA}(m) \prod_{i=1}^r 1_{\cA}( p_i\nn(n) + m )  
	\\& = \EE_{m \leq M} \EE_{n \leq N} S^m f_0(1_\cA)  \prod_{i=1}^r S^{p_i\nn(n) + m} f_0( 1_{\cA} )
	\\& \overset{(1)}\approx \int_{X_{\cA}} f_0 \cdot \EE_{n \leq N}  \prod_{i=1}^r S^{p_i\nn(n)} f_0 \d \mu_{\cA} 
	\\&
	 \overset{(2)}\approx \left( \int_{X_{\cA}} f_0(x) \d \mu_{\cA} \right)^{r+1} = \densNat(\cA)^{r+1}.
\end{align*}

The approximation labelled $(1)$ is simply the ergodic theorem, and is valid as long as $M$ is sufficiently large, with $N$ fixed. The key difficulty lies in making precise and justifying step (2), which will involve understanding the convergence of averages such as the one under the integral.

Study of similar averages was pioneered by Bergelson in \cite{Bergelson1987}, but without the dependence of the polynomials on $N$. We shall call a sequence of $r \geq 1$ polynomials $(p_i)_{i=1}^r$ \emph{admissible} if none of $p_i$ and $p_i - p_j$ with $i \neq j$ are constant.
\begin{theorem}[Bergelson \cite{Bergelson1987}]\label{ergo::thm:PET-Bergelson}
	Suppose that a m.p.s.\ $\fX = (X,T,\Borel,\mu)$ is weakly mixing and invertible. 
	Let $(p_i)_{i=1}^r$ be an admissible sequence of polynomials. Let $f_i \in L^\infty(\mu)$ for $i \in [r]$. Then:
	\begin{equation}
	\EE_{n \leq N} \prod_{i=1}^r T^{p_i(n)} f_i \xrightarrow[N \to \infty]{L^2} \prod_{i=1}^r \int f_i d\mu 
	. 
	\end{equation}
\end{theorem}

Here, we need a slight variation of the above theorem, already mentioned in the introduction. Refining the notion of admissibility, we shall call a family of sequences of polynomials $(p\tt_i)_{i=1}^r$ \emph{uniformly admissible} if for any $i$, $\deg(p\tt_i) > 0$ and $\lc(\deg(p\tt_i)$ are independent of $t$, and if likewise for any $i \neq j$, $\deg(p\tt_i -  p\tt_j) > 0$ and $\deg(p\tt_i -  p\tt_j)$ are independent of $t$. (Here, $t$ runs over some unspecified index set $I$.) For example, the family $(x^2+a\tt_1, x^2 + x + a\tt_2)$ is uniformly admissible, but $(x^2+b\tt_1 x + a\tt_1, x^2 + x + a\tt_2)$ is, in general, not (unless $b\tt_1$ is independent of $t$ and $b\tt_1 \neq 1$).

\begin{theorem}\label{ergo::thm:PET-uniform-1}
	Suppose that a m.p.s.\ $\fX = (X,T,\Borel,\mu)$ is weakly mixing and invertible. 
	Let $(p\tt_i)_{i=1}^r$ be a uniformly admissible family of sequences of polynomials, indexed by $t \in I$. Let $f_i \in L^\infty(\mu)$ for $i \in [r]$. Then:
	\begin{equation}
	\EE_{n \leq N} \prod_{i=1}^r T^{p\tt_i(n)} f_i \xrightarrow[N \to \infty]{L^2} \prod_{i=1}^r \int f_i d\mu 
	,\quad\text{uniformly in $t$.} 
	\end{equation}
\end{theorem}

Before embarking upon the proof of the above theorem, we explain how it completes the proof of the first of our main results. 
\begin{proof}[Proof of Theorem \ref{thm:main-uniform} assuming Theorem \ref{ergo::thm:PET-uniform-1}]
	Recall that $\cA$ takes the form $\cA = \set{n \in \NN_0}{ f(T^n x) = 1}$ where $\fX = (X,T,\Borel,\mu)$ is a weakly mixing m.p.s.\ and $x_0$ is a generic point for the $0,1$-valued function $f \in L^\infty(\mu)$ with $\int f d\mu = \densNat(\cA)$. Without loss of generality, we may assume that $\fX$ is invertible.
	 
	Suppose, for the sake of contradiction, that the system
\begin{equation}
	p_1\nn(n) + m \in \cA, \quad \dots, \quad p_r\nn(n) + m \in \cA, \quad m \in \cB 
	\label{eq:002}
\end{equation}
	has no solution, and consider the quantity $$\mathcal{L}(N) = \limsup_{M \to \infty} \abs{\EE_{m \leq M} \EE_{n \leq N} 1_{\cB}(m) \bra{\prod_{i=1}^r 1_{\cA}(m+p_i\nn(n)) - \densNat(\cA)^r }},$$ 
	which can be construed as the (normalised) deviation of the number of solutions to \eqref{eq:002} from the expected value of $\densNat(\cB) \densNat(\cA)^r MN $. On one hand, since \eqref{eq:002} lacks solutions, we have $\mathcal{L}(N) = \densNat(\cB) \densNat(\cA)^r$. On the other hand, we may approximate, with the use of Cauchy-Schwartz and the ergodic theorem:
\begin{align*}
\mathcal{L}(N) & \leq \lim_{M \to \infty} 
\bra{ 
\EE_{m \leq M} 
\bra{ \EE_{n \leq N} \prod_{i=1}^r 1_{\cA}(m+p_i\nn(n)) - \densNat(\cA)^r }^2
}^{1/2}
\\ &= 
\bra{ \int_{X} \bra{ \EE_{n \leq N} \prod_{i=1}^r T^{p_i\nn(n)}f   - \densNat(\cA)^r }^2 d\mu }^{1/2} 
\\ &= \norm{  \EE_{n \leq N} \prod_{i=1}^r T^{p_i\nn(n) } f - \densNat(\cA)^r }_{L^2(\mu)}.
\end{align*}

	An application of Theorem \ref{ergo::thm:PET-uniform-1} gives, in particular: 
	$$\EE_{n \leq N}  \prod_{i=1}^r T^{p_i\nn(n)} f \xrightarrow[N \to \infty]{L^2} \densNat(\cA)^{r}.$$
	
	Thus, if $N$ is large enough, then we conclude that $\mathcal{L}(N) < \densNat(\cB) \densNat(\cA)^r$, which is the sought for contradiction.
\end{proof}
\begin{remark}
	In the case $\cA = \cB$, a similar reasoning gives the asymptotic formula for the number of solutions $\cN(N,M)$ mentioned at the beginning of this section: $$\cN(N,M) = \densNat(\cA)^{r+1} MN (1+o_{N\to \infty} (1) + o_{N;M\to \infty}(1)).$$
\end{remark}

The remainder of this section is devoted to proving Theorem \ref{ergo::thm:PET-uniform-1}. In our argument, we follow the approach of Bergelson rather closely, taking care to account for uniformity of convergence. 
We will need an uniform version of van der Corput Lemma, which is a slight variation on the usual statement. We include the proof, which is rather standard, in the appendix, for the convenience of the reader.

\newcommand{\cH}{\mathcal{H}}
\begin{lemma}[Uniform van der Corput]\label{ergo::lem:vdCorput}
	Suppose that $(u\tt_n)_n$ is a sequence of vectors in a Hilbert space $\cH$ with $\norm{u\tt_n} \leq 1$, indexed by $t \in I$. Suppose further that for $s\tt_h \in \RR_{\geq 0}$ we have:
	\begin{equation}
		\abs{ \EE_{n \leq N} \spr{u\tt_n}{u\tt_{n+h}} } \leq s\tt_h + o_{N \to \infty}(1),
	\end{equation}	 
	where the error term is uniform in $t$. Suppose further that:
	\begin{equation}
		\EE_{h \leq H} s\tt_h \xrightarrow[H \to \infty]{} 0
	,\quad\text{uniformly in $t$.} 
	\end{equation}	 
	Then we have:
	\begin{equation}
	\EE_{n \leq N} u\tt_n \xrightarrow[N \to \infty]{} 0
	,\quad\text{uniformly in $t$.} 
	\end{equation}	 
\end{lemma}

\begin{proof}[Proof of Theorem \ref{ergo::thm:PET-uniform-1}, linear case]

We will now deal with the case of Theorem \ref{ergo::thm:PET-uniform-1} when $\deg p_i\tt = 1$ for all $i$. In this case, $p_i\tt$ are necessarily of the form $p_i\tt(x) = a_i n + b\tt_i$, where $a_i$ are distinct integers which do not depend on $t$. 

We may assume without loss of generality that for each $i$, we have $\int f_i d\mu = 0$. Indeed, if this is not the case, we may simply replace the original functions by $\tilde{f_i} := f_i - \int f_i d\mu$. Likewise, we may assume $\norm{f_i}_\infty \leq 1$, else we may rescale.

The case $r = 1$, when there is only one polynomial, is simple. Indeed, we then have:
$$ \EE_{n \leq N} T^{p_1\tt(n)} f_1 = T^{b\tt_i} \left(  \EE_{n \leq N} T^{a_1 n} f_1 \right).$$
The average in the brackets does not depend on $t$, and converges to $0$ in $L^2$ which follows e.g.\ from Theorem \ref{ergo::thm:PET-Bergelson}. Since $T^{b\tt_i}$ preserves the $L^2$ norm, we have:
$$ \EE_{n \leq N} T^{p_1\tt(n)} f_1 \xrightarrow[N \to \infty]{L^2} 0, \text{ uniformly in $t$}.$$

For $r \geq 2$ we proceed by induction using van der Corput Lemma. Let us write $u_n := \prod_{i=1}^r T^{p_i\tt(n)} f_i$. We have:
\begin{align*}
	\EE_{n \leq N}\spr{u_n\tt}{u_{n+h}\tt} &= \int \prod_{i=1}^r T^{a_in + b_i\tt} \left( f_i \cdot T^{a_ih}f_i\right) d\mu \\
	&= \int \tilde f_{r,h} \cdot \EE_{n \leq N} \prod_{i=1}^{r-1} T^{\tilde a_i n + \tilde b_i\tt} \tilde f_{i,h} d\mu
\end{align*}
where we define:
\begin{align*}
	\tilde{a}_i := a_i - a_r, \quad \tilde{b}_i\tt := b_i\tt - b_r\tt, \quad \tilde f_{i,h} = f_i \cdot T^{a_i h} f_i. 
\end{align*}
The sequence of polynomials $\tilde p_i\tt(x) = \tilde a_i x + \tilde b_i\tt$ for $i \in [r-1]$ is uniformly admissible, since $\tilde p_i\tt- \tilde p_j\tt = p_i\tt - p_j\tt$. Hence, we can invoke the inductive assumption to conclude that:
$$ \EE_{n \leq N} \prod_{i=1}^{r-1} T^{\tilde a_i n + \tilde b_i\tt} \tilde f_{i,h} = 
	 \prod_{i=1}^{r-1} \int \tilde f_{i,h} d\mu + o_{h;N\to \infty}(1),
$$
with the error term independent of $t$. Letting $s_h = \abs{\prod_{i=1}^{r} \int \tilde f_{i,h}d \mu}$ we have:
$$
\abs{
\EE_{n \leq N}\spr{u_n\tt}{u_{n+h}\tt}
} \leq s_h + o_{h;N\to \infty}(1).
$$
By a standard argument, we have $\EE_{h \leq H} s_h \to 0$ as $H \to \infty$, and convergence is automatically uniform in $t$, since $s_h$ does not depend on $t$. We are now in position to apply Lemma \ref{ergo::lem:vdCorput} to conclude that $\EE_{n \leq N} u_n\tt \to 0$ in $L^2$ as $N \to \infty$, uniformly in $t$. This finishes the inductive step, and thus the proof in the case $\deg p_i\tt = 1$.

\end{proof} 
\section{PET induction}\label{sec:PET}

To prove the general case of Theorem \ref{ergo::thm:PET-uniform-1} we will use PET induction. We will now introduce the key concepts in separation from the proof of this particular result.

\begin{definition}[Characteristic vector]
	Let $p = (p_i)_{i=1}^r$ be the a sequence of polynomials (not necessarily admissible). For $k \in \NN$, let $F_k$ be the set of those $p_i$ with $\deg p_i = k$. We define the \emph{characteristic vector} of $p$, which we denote by $\charX {p}$, by declaring $\charX {p}_k$ to be equal to the number of different leading coefficients $\lc(p)$ for $p \in F_k$.
\end{definition}

It does not matter much if we define characteristic vectors to have entries for all $k$ or just for $k \leq \max_i \deg p_i$. For the sake of simplicity, we assume the former. We make the set of possible characteristic vectors (i.e.\ $\NN_0$-valued sequences with finitely many non-zero entries) into an ordered set by introducing reverse-lexicographical order. Recall that $\charN > \charN'$ if for the largest $k$ with  $\charN_k \neq \charN'_k$ we have  $\charN_k > \charN'_k$. It is well known fact that $\NN_0^{\NN}$ is well-ordeder by the reverse lexicographical order. Thus, any decreasing sequence $\charN > \charN' > \charN'' > \dots$ has to be finite.

For the inductive step, we shall need the following operation. Let $p=(p_i)_{i=1}^r$ be a sequence of polynomials. We may, without loss of generality, assume that $\deg p_1 \geq \deg p_2 \geq \dots \geq  \deg p_r$. We then define the sequence $\tilde{p}_h$ to be the concatenation of two sequences:
\begin{align} 
	\tilde{p}_{0,h,i}(x) &:= p_i(x) - p_r(x),\ &&i \in [r-1] \label{ergo::eq:def-p-1hi} \\
	\tilde{p}_{1,h,i}(x) &:= p_i(x+h)- p_r(x),\ &&i \in [r]. \label{ergo::eq:def-p-2hi}
\end{align}

The following statements give base for the PET induction. We cite it here merely as a list of facts. For proofs, which are not difficult, we refer the reader to~\cite{Bergelson1987}.
\begin{fact} \label{ergo::fact:char-vect-props-Bergelson}
Let $p = (p_i)_{i=1}^r$ be an admissible sequence of polynomials.
\begin{enumerate}
\item If $\chi(p)_1 = 0$, then $\tilde p_h$ is admissible for all but finitely many $h$.
\item The characteristic vector $\charX {\tilde{p}_h}$ takes the same value for all but finitely many values of $h$.
\item We have $\charX {\tilde{p}_h} < \charX {p}$. 
\end{enumerate}
\end{fact}

Essentially the same statement is true with admissible sequences of polynomials replaced by a uniformly admissible families of sequences.

\begin{lemma}\label{ergo::fact:char-vect-props-uniform}
 Let $p\tt = (p\tt_i)_{i=1}^r$ be a uniformly admissible sequence of polynomials.
\begin{enumerate}
\item\label{ergo::cond:O@char-vect-uni} The characteristic vector $\charX {p\tt}$ does not depend on $t$.
\item\label{ergo::cond:A@char-vect-uni} For all but finitely many $h$, the degrees and leading coefficients of $\tilde p\tt_{h,k}$ and $\tilde p\tt_{h,k} - \tilde p\tt_{h,l}$ do not depend on $t$.
\item\label{ergo::cond:B@char-vect-uni} For all but finitely many $h$, $\charX {\tilde{p}\tt_h}$ does not depend on $h$ and $t$, and we have $\charX {\tilde{p}\tt_h} < \charX {p\tt}$.
\item\label{ergo::cond:C1@char-vect-uni} If $\charX {p\tt}_1 =0$ then for all but finitely many $h$, $\tilde{p}\tt_h$ is uniformly admissible.
\item\label{ergo::cond:C2@char-vect-uni} If $\charX {p\tt}_1 \geq 1$ then for all but finitely many $h$, we have for any $(\sigma,i) \neq (\rho,j)$ ($\sigma,\rho = 0$ or $1$, $i,j \in [r]$ or $[r-1]$ accordingly) that $\deg( \tilde p\tt_{\sigma,h,i} - \tilde p\tt_{\rho,h,j} ) \geq 1$, except when $i = j$ and $\deg p\tt_i = 1$.	 
\end{enumerate}
\end{lemma}

\begin{proof}
Item \eqref{ergo::cond:O@char-vect-uni} is clear, since $\charX {p\tt}$ depends only on leading coefficients and degrees of polynomials in $p\tt$, and these are independent of $t$.

For item \eqref{ergo::cond:A@char-vect-uni}, we first deal with leading terms of polynomials in $\tilde p\tt_h$. These are either of the form $p\tt_i(x) - p\tt_r(x)$, or of the form $p\tt_i(x+h) - p\tt_r(x)$. In the former case, the leading term does not depend on $t$ by assumption. In the latter case, we reduce to the former unless $\deg p_i\tt = \deg p_r\tt$ and $\lc p_i\tt = \lc p_r\tt$. When this happens, we can write $p_i\tt(x) = a x^d + b_1\tt x^{d-1} + q_1\tt(x)$ and $p_r\tt(x) = a x^d + b_2\tt x^{d-1} + q_2\tt(x)$, where $a$ and $b:=b_1\tt-b_2\tt$ are independent of $t$. It follows that $\lc ( p_i\tt(x+h) - p_r\tt(x)) = d a h + b$, except for at most one value of $h$, when this is $0$.

Secondly, we deal with leading terms of differences. They are of the form $p_{i}(x+\sigma h ) - p_j(x + \rho h)$ with $\sigma, \rho \in \{0,1\}$. In the case when $\rho = 0$, we use essentially the same argument as before. The case $\rho = 1,\ \sigma = 0$ follows by the same argument as the case $\rho = 0,\ \sigma = 1$. Finally, the case $\rho = 1,\ \sigma = 1$ follows from the case $\rho = 0,\ \sigma = 0$ by a change of variable.

For item \eqref{ergo::cond:B@char-vect-uni}, we notice that the argument in  \eqref{ergo::cond:A@char-vect-uni} shows that $\charX {\tilde p\tt_h}_k = \charX {p\tt}_k$ for all $k > \deg p_r$. It will suffice to check that for $k = \deg p_r$ we have $\charX {\tilde p\tt_h}_k < \charX {p\tt}_k$. This follows, because each $p_i\tt$ with $\deg p_i\tt = \deg p_r\tt$ contributes $1$ to $\charX {\tilde p\tt_h}_k$, except for $p_r\tt$ itself.

For items \eqref{ergo::cond:C1@char-vect-uni} and \eqref{ergo::cond:C2@char-vect-uni}, we notice that the only condition that remains to be checked to verify that $\tilde p_h\tt$ are uniformly admissible for almost all $h$ is that the differences $p\tt_{\sigma,h,i}(x) - p\tt_{\rho,h,j}(x)$ should be non-constant. The only possible degrees of such difference  are $\deg(p_i\tt-p_j\tt)$ (if the leading terms differ or $\sigma = \rho$), or $\deg p_i\tt - 1$ (otherwise). In the former case, we have $\deg(p_i\tt-p_j\tt) > 0$ by assumption. In the latter case, we have  $\deg p_i\tt - 1 > 0$, unless $\deg p_i\tt = 1$.
\end{proof}
With this machinery, we are ready to complete the proof.
\begin{proof}[Proof of Theorem \ref{ergo::thm:PET-uniform-1}, general case]
We proceed by induction of $\charN = \charN(p\tt)$. Because the set of all characteristic vectors is well-ordered, we may assume that the claim of the theorem is true for any $p^{\prime(t)}$ with $\charN(p^{\prime(t)}) < \charN$. 
 We have already dealt with all $\chi < (0,1,0,\cdots)$.

Take any $\charN \geq (0,1,0,\cdots)$, any uniformly admissible $p\tt = (p\tt_i)_{i=1}^r$ such that $\charX {p\tt} = \charN$, and let $f_i \in L^\infty(\mu)$. As before, we may assume that $\int f_i d\mu = 0$ and $\norm{f_i}_{\infty}\leq 1$ for all $i$. We need to show that
	\begin{equation*}
	\EE_{n \leq N} \prod_{i=1}^r T^{p\tt_i(n)} f_i \xrightarrow[N \to \infty]{L^2} 0,\quad
	\mbox{uniformly in $t$.}  
	\end{equation*}
	
	Let $u_n := \prod_{i=1}^r T^{p\tt_i(n)} f_i$. Bearing in mind that we hope to apply van der Corput Lemma, we compute:
	\begin{align}
	\EE_{n \leq N} \spr{u_n}{u_{n+h}}
	&= \EE_{n \leq N}  \int \prod_{i=1}^r T^{p\tt_i(n)} f_i 
		\cdot \prod_{i=1}^r T^{p\tt_i(n+h)} f_i d\mu \\
	&= \int f_r \cdot \EE_{n \leq N} \left( \prod_{i=1}^{r-1} T^{\tilde p\tt_{0,h,i}(n)} f_i 
		\cdot \prod_{i=1}^r T^{\tilde p\tt_{1,i,h}(n)} f_i \right) d\mu 
	\label{ergo::eq:065}
	\end{align}
	with $\tilde{p}_{0,h,i}$ and $\tilde{p}_{1,h,i}$ defined as in \eqref{ergo::eq:def-p-1hi} and \eqref{ergo::eq:def-p-2hi}. 	Except for finitely many values of $h$, we have that $\charX {\tilde{p}_h} < \charN$ does not depend on $t$. 

	We now need to branch out into two cases. Suppose first that $\charX {p\tt}_1 = 0$. Then, $\tilde p_h\tt$ is uniformly admissible and $\charX { p_h\tt } < \charN$, so by the inductive assumption we may write:
	\begin{align}
	\EE_{n \leq N} \left( \prod_{i=1}^{r-1} T^{\tilde p\tt_{0,h,i}(n)} f_i 
		\cdot \prod_{i=1}^r T^{\tilde p\tt_{1,h,i}(n)} f_i \right) = o_{h;N \to \infty}(1).	
	\label{ergo::eq:066}	
	\end{align}
	The decay rate implicit in the $o$-notation is independent of $t$. Hence, the assumptions of van der Corput Lemma \ref{ergo::lem:vdCorput} are satisfied with $s_h = 0$ for all but finitely many $h$. Application of the lemma gives precisely the sought convergence. 
	
	Secondly, suppose that $\charX {p\tt}_1 \neq 0$. In this case, let $s$ denote the number of linear polynomials among $p_i\tt$, and let $r' = r-s+1$. We will adapt the argument from the linear case. If the linear polynomials in $p\tt$ are given by $p_i\tt(x) = a_i x + b_i\tt$ then the expression under the integral in \eqref{ergo::eq:065} becomes:
	$$
	\tilde f_r \cdot \EE_{n \leq N} \left( \prod_{i=1}^{r'} T^{\tilde p\tt_{0,h,i}(n)} f_i 
		\cdot \prod_{i=1}^{r'} T^{\tilde p\tt_{1,h,i}(n)} f_i \right)	
		\cdot \prod_{i=r'+1}^{r-1} T^{\tilde a_i n + \tilde b_i\tt} \tilde f_{i,h}
		,
	$$
	where $\tilde a_i = a_i - a_r$, $b_i\tt = b_i\tt - b_r\tt$ and $\tilde f_{i,h} = f_i \cdot T^{a_i h} f_i$. We may now apply the inductive assumption to the uniformly admissible sequence $\tilde p^{\prime (t)}_h$, which is the concatenation of $( \tilde p_{0,h,i}\tt(x) )_{i \in [r']}$, $( \tilde p_{1,h,i}\tt )_{i \in [r']}$ and $( \tilde a_i x + \tilde b_i\tt )_{i \in (r',r]}$. Note that $\charX { \tilde p^{\prime (t)}_h } < \charX { \tilde p\tt_h } < \charN$ and that $r' > 0$ because $\charN  > (0,1,0,\dots)$. Thus, we recover the bound from \eqref{ergo::eq:066}, and the rest of the argument proceeds in the same way.
	
\end{proof} 
\section{Doubly polynomial averages}\label{sec:Polynomial}
\newcommand{\Folner}{F\o{}lner}
\newcommand{\T}{M}
\newcommand{\tto}[2]{ \xrightarrow[#1]{#2} }

We now deal with polynomial families of polynomials, such as the ones which appear in Theorem \ref{thm:main-polynomial}. Our first step is again to translate the problem into a question about  convergence of certain polynomial averages. Fortunately, we will be able to essentially reduce the problem to known results on $L^2$ convergence of polynomial averages along \Folner\ sequences.

In this section, rather than uniform convergence we dealt with in Section \ref{sec:Uniform}, we will be interested only in convergence of averages such as 
$
	\EE_{n \leq N(t)} \prod_{i=1}^r T^{p\tt_i(n)} f_i,
$
 as $t \to \infty$ for a specific sequence $N(t)$. We can afford to be quite flexible in choice of $N(t)$; the only conditions we need to impose are 
\begin{equation}
\label{cond:POLY:N(t)}
N(t) \tto{t\to\infty}{} \infty, \qquad \frac{N(t) - N(t+h)}{N(t)} \tto{t\to \infty}{} 0, \text{ for any } h \in \ZZ.
\end{equation}

As defined in the introduction, a polynomial family of polynomial sequences $p_i\tt(x) \in \ZZ[x,t]$, $i \in [r]$, is \emph{admissible} if $\deg_x p_i\tt(x) >0$ and $\deg_x (p_i\tt(x)- p_j\tt(x)) > 0$ for all $i \neq j$.

\begin{theorem}\label{ergo::thm:PET-uniform-2}
	Suppose that a m.p.s.\ $\fX = (X,T,\Borel,\mu)$ is weakly mixing and invertible, and that the sequence $N(t)$ obeys \eqref{cond:POLY:N(t)}. 
	Let $(p\tt_i)_{i=1}^r$ be an admissible polynomial family of sequences of polynomials. Let $f_i \in L^\infty(\mu)$. Then there exists a set $S \subset \NN$ with density $\densNat(S) = 1$ such that
	\begin{equation}
	\EE_{n \leq N(t)} \prod_{i=1}^r T^{p\tt_i(n)} f_i \xrightarrow[t \to \infty, t \in S]{L^2} \prod_{i=1}^r \int f_i d\mu.
	\label{ergo::eq:conv@thm:PET-uniform-2}
	\end{equation}
\end{theorem}

\begin{proof}[Proof of Theorem \ref{thm:main-polynomial} assuming Theorem \ref{ergo::thm:PET-uniform-2}]
	Clearly, $N(t) = t$ satisfies \eqref{cond:POLY:N(t)}. The argument follows by a direct repetition of the proof of Theorem \ref{thm:main-uniform} in Section \ref{sec:Uniform}.
\end{proof}

\begin{remark}
	It is an immediate consequence of Bergelson's Theorem \ref{ergo::thm:PET-Bergelson} that the conclusion of Theorem \ref{ergo::thm:PET-uniform-2} also holds when the sequence $N(t)$, instead of obeying \eqref{cond:POLY:N(t)}, is sufficiently steeply increasing. This leaves open an interesting gap. It is possible that Theorem \ref{ergo::thm:PET-uniform-2} holds when \eqref{cond:POLY:N(t)} is replaced with the weaker condition $N(t) \to \infty$ as $t \to \infty$. In Proposition \ref{ergo::thm:PET-uniform-20} we verify this for a single linear polynomial. \end{remark}

We devote most of the rest of this section to proving Theorem \ref{ergo::thm:PET-uniform-2}. To begin with, we cite a simple lemma, which allows us to conveniently reformulate the problem.

\begin{lemma}\label{ergo::lem:caesaro-conv-unif}
Let $(a_n)_{n=1}^\infty$ be a sequence with $a_n \in [0,1]$. Then the following conditions are equivalent:
\begin{enumerate}
	\item We have convergence $ \EE_{n \leq N} a_n \xrightarrow[N \to \infty]{} 0 $.
	
	\item There exists $J \subset \NN$ with $\densNat(J) = 1$ such that $a_n \xrightarrow[n \to \infty, n \in J]{} 0$.
\end{enumerate}
\end{lemma}
\begin{proof}
	See \cite[Chpt 2.7]{EinsiedlerWard}
\end{proof}

In the situation of Theorem \ref{ergo::thm:PET-uniform-2} we may always assume that $\int_X f_i d\mu = 0$ and $\norm{f_i} \leq 1$ for each $i$. With this assumption, Theorem \ref{ergo::thm:PET-uniform-2} will follow by Lemma \ref{ergo::lem:caesaro-conv-unif} if we are able to show that
	\begin{equation}
	\EE_{t \leq \T} \norm{ \EE_{n \leq N(t) } \prod_{i=1}^r T^{p\tt_i(n)} f_i}_{L^2(\mu)}^2 \tto{\T\to \infty}{} 0.
	\label{eq:POLY:001}
	\end{equation}
	
Expanding and using Cauchy-Schwartz, we conclude that \eqref{eq:POLY:001} will follow from
	\begin{align}
	\EE_{(t,n,m) \in \Phi(\T) }
		\prod_{i=1}^r T^{p\tt_i(n)} f_i 
		\cdot \prod_{i=1}^r T^{p\tt_i(m)} f_i \tto{\T\to \infty}{} 0,
	\label{eq:POLY:002}
	\end{align}
where the average is being taken over the set
\begin{align}
	\Phi(\T) = \set{(t,n,m) \in \ZZ^3}{1 \leq t \leq \T,\ 1 \leq n \leq N(t),\ 1 \leq m \leq N(t)}.\label{eq:POLY:003-def-of-Phi}
\end{align}

Recall that a sequence $(\Psi(\T))_{\T=1}^\infty$ of finite subsets of an abelian group $G$ is a \emph{\Folner\ sequence} if for any $h \in G$ it holds that 
$$
	\frac{ \abs{ (\Psi(\T) + h) \triangle \Psi(\T)} }{ \abs{\Psi(\T) } } \tto{\T\to \infty}{} 0,
$$
where $\triangle$ denotes the symmetric difference.
\begin{observation}
	If the sequence $N(t)$ obeys condition \eqref{cond:POLY:N(t)}, then the sequence $\Phi(\T)$ defined by \eqref{eq:POLY:003-def-of-Phi} is \Folner.
\end{observation}
\begin{proof}
	Fix a choice of $(h,a,b) \in \ZZ^3$. It will be convenient to assume a probabilistic perspective: we choose $(t,n,m) \in \Phi(\T)$ uniformly at random, and show that asymptotically almost surely as $\T\to \infty$ (shortened to a.a.s.) we have $(t+h,n+a,m+b) \in \Phi(\T)$.  
	
	It is easy to check that for any constants $A,B$ it holds a.a.s.\ that $A \leq t \leq \T-B$. In particular, a.a.s.\ $1 \leq t+h \leq \T$, and also $1 \leq n+a, m+b$. For any $s$, we may then estimate
	\begin{align*}
		\PP\left( n+a > N(t+h) \middle|\vphantom{a^b} t = s \right) 
		&\leq \PP\left( n+a > N(t) \middle|\vphantom{a^b} t = s\right) 
		\\ &+ \PP\left(N(t) \geq n+a > N(t+h) \middle|\vphantom{a^b} t=s \right)  
		\\ &\leq \frac{a}{N(s)} + \frac{N(s)-N(s+h)}{N(s)}
	\end{align*}
	For any choice of $\e>0$, it is true a.a.s.\ that $a/N(t) < \e$ and $\frac{N(t)-N(t+h)}{N(s)} < \e$; hence a.a.s.\ $n+a \leq N(t+h)$. By a symmetric argument, a.a.s.\  $m+b \leq N(t+h)$.
\end{proof}

	Averages such as \eqref{eq:POLY:002}, or more generally of the form
	\begin{equation}
		\EE_{n \in \Psi(\T)} \prod_{i=1}^r T^{q_i(n)} f_i,
	\label{eq:average}
	\end{equation}
	where $\Psi(\T)$ is a \Folner\ sequence in $\ZZ^d$ and $q_i \colon \ZZ^d \to \ZZ$ are polynomials, are well studied. To prove convergence of these averages to $\prod_{i=1}^r \int f_i d\mu$, one can in principle apply a standard PET induction argument reminiscent of that in \cite{Bergelson1987}. Unfortunately, this result is not recorded in the literature, but we can see it as a special case of a much stronger theorem.
	
	In a larger generality, averages such as \eqref{eq:average} have been studied by Leibman without the assumption that the system should be weakly mixing. Crucially, one can show that they converge in $L^2$. Even though, unlike the case of the classical ergodic theorem of von Neumann, the limit function will not in general be $T$-invariant, we have a convenient description of the limit in terms of the Host-Kra factors, which we will discuss shortly.

\newcommand{\bb}[1]{\underline{#1}}

	\begin{theorem}[Leibman \cite{Leibman2005b}]
	Suppose that a m.p.s.\ $\fX = (X,T,\Borel,\mu)$ is invertible. Let $\Psi(\T)$ be a \Folner\ sequence, and let $q_i \colon \ZZ^d \to \ZZ$ be polynomials and $f_i \in L^\infty(\mu)$ for $i \in [r]$. Then, the averages
	\begin{equation}
		\EE_{\bb n \in \Psi(\T)} \prod_{i=1}^r T^{q_i(\bb n)} f_i
		\label{eq:POLY:004-averages-Folner-poly}
	\end{equation}
	converge in $L^2(X)$ as $\T\to \infty$. Moreover, supposing that $\deg q_i > 0$ and $\deg( q_i - q_j) > 0$ for each $i \neq j$, there exists an integer $k$, dependent only on $r$ and the maximal degree $\max_{i \in [r]} \deg q_i$, such that the Host-Kra factor $\fZ_k$ of $\fX$ is characteristic for convergence of the averages \eqref{eq:POLY:004-averages-Folner-poly}, in the sense that 
	$$
	\lim_{\T\to \infty} \EE_{\bb n \in \Psi(\T)} \prod_{i=1}^r T^{q_i(\bb n)} f_i = 
	\lim_{\T\to \infty}  \EE_{\bb n \in \Psi(\T)} \prod_{i=1}^r T^{q_i(\bb n)} \EE(f_i | \fZ_k).
	$$
	\end{theorem}

In fact, the convergence of \eqref{eq:POLY:004-averages-Folner-poly} is established by verifying convergence under the additional assumption that $\fX$ is a nilsystem \cite{Leibman2005a} (or indeed an inverse limit of nilsystems), and then combining the fact that $\fZ_k$ is characteristic for \eqref{eq:POLY:004-averages-Folner-poly} with the fact that $\fZ_k$ is an inverse limit of nilsystems \cite{HostKra2005}. We note that the second part of this theorem is not stated in this precise language in \cite{Leibman2005b}, but our restatement is an immediate consequence of Theorem 3 therein, by a standard telescoping argument. 

The Host-Kra factors are introduced in \cite{HostKra2005} (see also \cite{Ziegler2007}), and are most easily described in terms of the Host-Kra-Gowers norms $\normHK{\cdot}_k$. Define inductively $\normHK{f}_0 = \int f d\mu$ for $f \in L^\infty(\mu)$, and for $k \geq 0$
$$
	\normHK{f}_{k+1}^{2^{k+1}} = \lim_{N \to \infty} \EE_{n \leq N} \normHK{f \cdot T^n f}_{k}^{2^{k}},
$$
where $\normHK{f}_{k+1} \geq 0$. It can be checked that this definition is well posed, and indeed defines a norm for $k \geq 1$. The factor $\fZ_k$ is then characterised by the property that $\EE(f | \fZ_k) = 0$ if and only if $\normHK{f}_{k+1} = 0$, where $f \in L^2(\mu)$. For the purpose of this paper, the key point is that for weakly mixing systems, the Host-Kra factors are trivial.

\begin{fact}
	Suppose that the m.p.s.\ $\fX$ is weakly mixing. Then, for any $k$, the Host-Kra factor $\fZ_k$ of $\fX$ is trivial. In particular $\EE(f|\fZ_k) = \int_X f d\mu$ for each $f \in L^\infty(\mu)$. 
\end{fact}	
\begin{proof}
	This is mentioned e.g.\ in \cite{HostKra2005}. With implicit definition of $\fZ_k$ as above, one can prove by a simple induction that $\normHK{f}_{k}^{2^k} = \bra{ \int_X f d\mu }^{2^k}$ for each $k$, from which the claim easily follows.
\end{proof}

	\begin{corollary}\label{ergo::cor:Leibman-for-WM}
	Suppose that a m.p.s.\ $\fX = (X,T,\Borel,\mu)$ is invertible. Let $\Psi(\T)$ be a \Folner\ sequence, let $q_i \colon \ZZ^d \to \ZZ$ be polynomials with $\deg q_i > 0$ and $\deg(q_i - q_j) > 0$ for $i \neq j$, and $f_i \in L^\infty(\mu)$ for $i \in [r]$. Then, 
	\begin{equation}
		\EE_{\bb n \in \Psi(\T)} \prod_{i=1}^r T^{q_i(\bb n)} f_i \xrightarrow[\T\to \infty]{L^2} \prod_{i=1}^r \int_X f_i d\mu.
		\label{eq:POLY:005}
	\end{equation}
	\end{corollary}

We are now ready to finish our argument.

\begin{proof}[Proof of Theorem \ref{ergo::thm:PET-uniform-2}]
	Apply Corollary \ref{ergo::cor:Leibman-for-WM} to the averages \eqref{eq:POLY:002}, together with previous discussion in this section.\end{proof}

\begin{remark}
	Using the same argument, one can prove Theorem \ref{ergo::thm:PET-uniform-2} under a weaker assumption. Indeed, the condition that $\cA$ should be a weakly-mixing set can be replaced by the weaker requirement that $\cA$ should give rise to an ergodic m.p.s.\ $\fX_\cA$ via the construction in Section \ref{sec:Definitions}, and that $1_\cA$ should be orthogonal to the $k$-th Host-Kra factor of $\fX_\cA$, in the sense that $\EE(1_\cA - \densNat(\cA)1_{X_\cA} | \fZ_k) = 0$, for some sufficiently large $k$.
\end{remark}

\section{Concluding remarks}\label{sec:End}
We end with a slightly stronger version of Theorem \ref{ergo::thm:PET-uniform-2} for a single polynomial of degree $1$, as suggested in an earlier remark.

\begin{proposition}\label{ergo::thm:PET-uniform-20}
	Suppose that a m.p.s.\ $\fX = (X,T,\Borel,\mu)$ is weakly mixing and invertible, and let $p\tt(x) \in \ZZ[t,x]$ be polynomial with $\deg_x p\tt(x) =1$, and $f \in L^\infty(\mu)$. Then there exists a set $S \subset \NN$ with density $\densNat(S) = 1$ and
	\begin{equation}
	\EE_{n \leq N} T^{p\tt(n)} f \xrightarrow[N \to \infty]{L^2} \int f d\mu 
	,\quad\mbox{uniformly in $t \in S$.} 
	\label{ergo::eq:conv@thm:PET-uniform-20}
	\end{equation}
\end{proposition}
\begin{proof}
	Write $p\tt(x) = a(t) x + b(t)$. Since $T^{b(t)}$ is an isometry, without loss of generality we may assume that $b(t) = 0$. As usual, we may assume that $\int f d\mu = 0$ and $\norm{f}_\infty \leq 1$. If $a(t)$ is constant in $t$, we are done e.g.\ by Bergelson's theorem \ref{ergo::thm:PET-Bergelson}, so assume this is not the case.
	
		By the spectral theorem, there is a probability measure $\nu = \nu_f$ on $\TT$ such that:
	\begin{align*}
	\norm{ \EE_{n \leq N} T^{p\tt (n)} f }^2 
	= \norm{\EEl{n}{N} e\left(p(n)\a\right)}_{L^2(\TT,\nu)}^2
	&= \int_\TT \abs{\EEl{n}{N} e\left(a(t)n \a\right)}^2 d \nu(\a),
	\end{align*}
	where as usual $e(\a) = e^{2 \pi i \a}$.
	It is a well-known fact that since $\fX$ is weakly mixing, the measure $\nu$ has no atoms.	
	
	We have the elementary inequality $\abs{1 - e(\a)} \geq 4 \fpa{\a}$. Thus, if for some $\de > 0$ and $\a \in \TT$ we have $\fpa{ a \a} \geq \delta$ then:
	\begin{align}
	\abs{ \EEl{n}{N} e\left(an \a\right) } =  \abs{ \frac{1 - e(a N\a)}{ N (1 - e(a \a) ) } }^2 \ll \frac{1}{(N \de)^2}
	\label{ergo::eq:bound-001@thm:PET-uniform-2}
	\end{align}	
	where the implied constant is absolute (and equal to $\frac{1}{4}$).
	
	Let us denote by $\Gamma_{a,\de}$ the set of $\a \in \TT$ for which we have inequality $\fpa{a \a} < \de$. Using \eqref{ergo::eq:bound-001@thm:PET-uniform-2} for $\a \not \in \Gamma_{a, \de}$ and the trivial bound for $\a \in \Gamma_{\a,\de}$ we find:
	$$ 
		\int_\TT \abs{\EEl{n}{N} e\left(a n \a\right)}^2 d \nu(\a) \leq \nu\left(\Gamma_{a,\de}\right) + \frac{O(1)}{(N\de)^2} .
	$$
	We claim that there exists a set $S \subset \NN$ with $\densNat(S) = 1$ such that:
	\begin{align}
		\nu\left(\Gamma_{a(t),\de}\right) \xrightarrow[\de \to 0]{} 0, \quad \mbox{uniformly in $t\in S$.}
	\label{ergo::eq:bound-002@thm:PET-uniform-2}
	\end{align}
	Suppose that the claim has been established. We then have:
	\begin{align*}
	\norm{ \EE_{n \leq N} T^{p\tt (n)} f }^2 
	= o_{\de \to 0}(1) + o_{\de; N \to \infty}(1) = o_{N \to \infty}(1),
	\end{align*}
	with decay rates uniform in $t \in S$. This finishes the proof of the theorem. Hence, it remains to find $S$ with \eqref{ergo::eq:bound-002@thm:PET-uniform-2}. Our construction relies on the following observation.

	\begin{lemma}\label{ergo::obs-int@thm:PET-uniform-2}
	Given $\de$, there exists a set $S_\de$ with density $1$ such that
	\begin{align} \abs{ \nu\left(\Gamma_{a(t),\de}\right) - \lambda \left(\Gamma_{a(t),\de}\right)} \xrightarrow[t \to \infty,\ t \in S_\de]{} 0,
	\label{ergo::eq:measure-convergence@thm:PET-uniform-2}
	\end{align}
	where $\lambda$ denotes that Lebesgue measure (so $\lambda \left(\Gamma_{a(t),\de}\right) = 2 \delta$ for almost all $t$).
	\end{lemma}	
	\begin{proof}
	Because of the Lemma \ref{ergo::lem:caesaro-conv-unif}, it will suffice to prove that for fixed $\de > 0$ we have:
	 \newcommand{\pars}[1]{\left( #1 \right)}
	\begin{align*} 
	\EE_{t \leq \T} \pars{ \nu\left(\Gamma_{a(t),\de}\right) - \lambda \left(\Gamma_{a(t),\de}\right)}^2 \xrightarrow[\T\to \infty]{} 0,	
	\end{align*}
	 This will follow once we show that:
	 \begin{align*}
	  \EEl{t}{\T} \nu\left(\Gamma_{a(t),\de}\right) \xrightarrow[\T\to \infty]{}  \lambda \left(\Gamma_{a(t),\de}\right) ,\qquad
	 \EEl{t}{\T}  \nu\left(\Gamma_{a(t),\de}\right)^2 \xrightarrow[\T\to \infty]{}  \lambda \left(\Gamma_{a(t),\de}\right)^2. 
	 \end{align*}
For the first limit, we can rewrite:
	 \begin{align*}
	 \lim_{\T\to \infty} \EEl{t}{\T} \nu\left(\Gamma_{a(t),\de}\right)
	 &= \lim_{\T\to \infty} \int_\TT  \EEl{t}{\T} \chi_{\de}( a(t)\a ) d \nu(\a)
	 \\&= \int_\TT \lim_{\T\to \infty} \EEl{t}{\T} \chi_{\de}( a(t)\a ) d \nu(\a),
	  \end{align*}	 
	 where $\chi_{\de}$ denotes the characteristic function of the interval $(-\de,\de) \Mod{1}$, and the interchange of limit and the integral is justified by the dominated convergence theorem (assuming that the last limit exists).
	 
	For $\a \not \in \QQ$, since $\chi_\de$ is Riemann-integrable and $\a a(t)$ is polynomial with irrational coefficients, we have by a classical theorem of Weyl that:
	$$ \lim_{\T\to \infty} \EEl{t}{\T} \chi_{\de}( a(t)\a ) = \int_\TT \chi_{\de}(\a) d\lambda = \lambda(\Gamma_{a(t),\de}). $$	 
	As for $\a \in \QQ$, we know that $\nu$ has no atoms so $\nu(\QQ) = 0$. It follows that:
	 \begin{align*}
	 \lim_{\T\to \infty} \EEl{t}{\T} \nu\left(\Gamma_{a(t),\de}\right)
	 = \lambda(\Gamma_{a(t),\de}).
	  \end{align*}	 
For the second limit, we can proceed analogously. We first rewrite:
	 \begin{align*}
	 \lim_{\T\to \infty}  \EEl{t}{\T} \nu\left(\Gamma_{a(t),\de}\right)^2
	 &= \lim_{\T\to \infty} \int_{\TT \times \TT}  \EEl{t}{\T} 
	 \chi_{\de} ( a(t)\a ) \cdot \chi_{\de}  (a(t)\b ) d \nu(\a)d \nu(\b)
	 \\&= \int_{\TT \times \TT} \lim_{\T\to \infty}  \EEl{t}{\T} 
	 (\chi_{\de} \times \chi_{\de}) ( a(t)\a,\ a(t)\b ) d (\nu \times \nu)(\a,\b).
	 \end{align*}	 

	The set of $(\a,\b) \in \TT\times\TT$ which are linearly dependent over $\QQ$ has $\nu\times\nu$ measure $0$, since it is a union of countably many lines $\set{ (\a,\b)}{ k \a + l \b + m = 0}$ ($k,l,m\in\ZZ$) whose measure is $0$ by e.g.\ Fubini's theorem combined with $\nu$ having no atoms.

	If $(\a,\b)$ are $\QQ$-linearly independent, then by Weyl's theorem, the sequence $( a(t)\a,\ a(t)\b )$ is equidistributed. It follows that:
	$$\lim_{\T\to \infty}  \EEl{t}{\T} 
	 (\chi_{\de} \times \chi_{\de}) ( a(t)\a,\ a(t)\b ) = \int_{\TT \times \TT} (\chi_{\de} \times \chi_{\de}) d(\lambda \times \lambda) .$$
	As a consequence, we have the sought convergence:
	 \begin{align*}
	 \lim_{\T\to \infty} \EEl{t}{\T} \nu\left(\Gamma_{a(t),\de}\right)^2 = \left(\Gamma_{a(t),\de}\right)^2.
	\end{align*}	
	This finishes the proof of the lemma.
	\end{proof}	
	
\begin{observation}
	Let $\{ S_n \}_{n =1}^\infty$ be a sequence of sets $S_n \subset \NN$ with density $1$. Then there exists a single set $S \subset \NN$ with density $1$ such that for each $n$, $S \setminus S_n$ is finite.
\end{observation}
\begin{proof}
	We may assume without loss of generality that the family $S_n$ is descending, else we may replace $S_n$ by $\bigcap_{m \leq n} S_n$.
	We  define $S$ by declaring that $k \in S$ if and only if $k \in S_{n(t)}$ where $n(k)$ is an increasing function yet to be determined. If $n(k) \to \infty$ as $k \to \infty$, then clearly $S \setminus S_n$ is finite for any $n$. 
	It remains to check that if $n(t)$ increases sufficiently slowly, then $S$ has density $1$. This is a simple consequence of the fact that for each $n$, $S_n$ has density $1$.
\end{proof}	
	
	We are now ready to finish the proof of the proposition. Let $S_\de$ be the sets constructed in the above Lemma \ref{ergo::obs-int@thm:PET-uniform-2}. Let $S$ be a set with $\densNat(S) = 1$ and $S \setminus S_\de$ finite for each $\de \in \QQ$.	We then have for $t \in S$:
	\begin{align*}
		\nu\left(\Gamma_{a(t),\de}\right) = \lambda\left(\Gamma_{a(t),\de}\right) + o_{\delta;t \to \infty}(1).
	\end{align*}
	We need to show that for any $\e > 0$ one can find $\de(\e)$ such that for $\de < \de(\e)$ and any $t \in S$ we have 
	\begin{align}
	\nu\left(\Gamma_{a(t),\de}\right) < \e.
		\label{ergo::eq:bound-003@thm:PET-uniform-2}
	\end{align}
	
	Let $\de_0$ be such that $\lambda\left(\Gamma_{a(t),\de_0}\right) = 2 \de_0 < \e/2$. We can then find $t_0$ such that for $t > t_0$, $t \in S$ we have $ \nu\left(\Gamma_{a(t),\de_0}\right) < \e/2 + \e/2 = \e.$ Since $\nu\left(\Gamma_{a(t),\de}\right)$ is decreasing in $\de$, the bound \eqref{ergo::eq:bound-003@thm:PET-uniform-2} holds for any $\de < \de_0$, and $t > t_0,\ t \in S$.
	
	 On the other hand, for each $1 \leq t \leq t_0$, because $\bigcap_{\de > 0} \Gamma_{a(t),\de}$ is a finite set and $\nu$ is atomless, there is some $\delta_t > 0$ such that \eqref{ergo::eq:bound-003@thm:PET-uniform-2} holds for $\delta < \delta_t$.
	 
	 Taking $\de(\e) = \min_{t} \de_t$ with $t$ running over $[t_0] \cup \{0\}$ we find that \eqref{ergo::eq:bound-003@thm:PET-uniform-2} holds for all $\de < \de(\e)$. This finishes the proof of the claim.
\end{proof}

\appendix

\section{}

\begin{proof}[Proof of Lemma \ref{ergo::lem:vdCorput}]
	For fixed $H$ we have:
	$$ \norm{\EE_{n \leq N} \EE_{h \leq H} u\tt_{n+h}}^2 = \norm{ \EE_{n \leq N} u\tt_n}^2 + O\left(1/N\right).$$
	By Cauchy-Schwartz inequality, we can bound:
	$$ \norm{\EE_{n \leq N} \EE_{h \leq H} u\tt_{n+h}}^2 \leq \EE_{h,h' \leq H} \abs{ \EE_{n\leq N} \spr{u\tt_{n+h}}{u\tt_{n+h'}} } .$$
	For each summand above we have the bound:
	$$
	\abs{ \EE_{n\leq N} \spr{u\tt_{n+h}}{u\tt_{n+h'}} } = 
	\abs{ \EE_{n\leq N} \spr{u\tt_{n}}{u\tt_{n+\abs{h-h'}}} } +  o_{H; N \to \infty}\left(1\right)
	\leq s\tt_h + o(1).
	$$
	It follows that:
	$$ \norm{ \EE_{n \leq N} u\tt_n}^2 \leq \EE_{h,h' \leq H} s\tt_{\abs{h-h'}} + o_{H; N\to \infty}(1).$$
	Summing by parts gives:
	$$\EE_{h,h' \leq H} s\tt_{\abs{h-h'}} \leq \EE_{h\leq H} \frac{h}{H} \EE_{h' \leq h} s\tt_h + o_{H \to \infty}(1).$$
	Let $K \leq H$ be arbitrary. Splitting the above average into $h \leq K$ and $ K < h \leq H$ gives:
	$$ \EE_{h\leq H} \frac{h}{H} \EE_{h' \leq h} s\tt_h \leq \frac{K}{H} + \max_{K < h \leq H} s\tt_h
	= o_{K;H\to \infty}(1) + o_{K \to \infty}(1)	
	.$$
	It follows that:
	$$ \norm{\EE_{n \leq N} u\tt_{n}}^2 = o_{H;N\to \infty}(1) + o_{K;H\to \infty}(1) + o_{K \to \infty}(1),$$
	which implies the sought convergence.
\end{proof}
 
\bibliographystyle{abbrv}
\bibliography{bibliography}

\end{document}